\documentclass[12pt,reqno,a4paper]{amsart}  	
\usepackage[bottom = 4cm, top = 3.2cm, left = 3cm, right = 3cm]{geometry}

\usepackage{hyperref}

\usepackage{graphicx}				
\usepackage{yfonts}
\usepackage{tikz}
\usepackage{caption}
\usepackage{amsmath,amsfonts,amssymb,amsthm,epsfig,epstopdf,titling,url,array,mathtools}
\usepackage{sidecap}
\usepackage{float}
\usepackage[utf8]{inputenc}
\usepackage[english]{babel}
\usepackage{t1enc}
\usepackage{subcaption}
\usetikzlibrary{positioning,arrows,calc}								
\usepackage{wrapfig, framed, caption}
\usepackage[titletoc]{appendix}
\usepackage{enumitem}
\usepackage{faktor}
\usetikzlibrary{decorations.pathmorphing}

\def\F{\ensuremath\mathfrak{F}}
\def\Uf{\ensuremath\mathrm{Uf}}

\def\uf{\ensuremath\mathfrak{F}^{\mathfrak{ue}}}
\def\ue{\ensuremath\mathfrak{ue}}
\def\M{\ensuremath\mathfrak{M}}
\def\dom{\ensuremath\text{dom}}

\def\>{\ensuremath\rangle}
\def\<{\ensuremath\langle}

\newtheorem{thm}{\bfseries Theorem}[section]
\newtheorem{prop}[thm]{\bfseries Proposition}
\newtheorem{lem}[thm]{\bfseries Lemma}
\newtheorem{fact}[thm]{\bfseries Fact}
\newtheorem{cor}[thm]{\bfseries Corollary}
\theoremstyle{definition}  
\newtheorem{rem}[thm]{\bfseries Remark}
\newtheorem{defn}[thm]{\bfseries Definition}
\newtheorem{exmp}[thm]{\bfseries Example}

\usepackage{amssymb}

\usepackage{authblk}
\author{Zal\'{a}n Moln\'{a}r} 
\address{E\"otv\"os Lor\'and University, Department of Logic,  Budapest, Hungary}
\email{mozaag@gmail.com}

\title{Ultrafilter extensions of bounded graphs are elementary}
\date{}

\begin{document}
	\maketitle
	
	\noindent \textbf{Abstract.} The main motivation of this paper is the study of first-order model theoretic properties of structures having their roots in modal logic.  We will focus on the connections between  ultrafilter extensions and  ultrapowers.   We show that certain structures (called bounded graphs) are elementary substructures of their ultrafilter extensions, moreover  their modal logics coincide. \\
	
	\noindent \textbf{keywords:} ultrafilter extension, elementary substructures, bounded graphs.
	
	\section{Overview}
	
	There are two different approaches that are called ultrafilter extensions of structures (graphs) equipped with a single binary relation.  The recently introduced one has its origins in model theory \cite{Saveliev1}  and was further investigated in \cite{Saveliev0, Saveliev2, Saveliev3}, while the other  emerged from modal logics and universal algebra and was first introduced in \cite{jonsson}. The connections between  the extensions were initially investigated in \cite{Saveliev4}.  In this paper, we will focus on the second type of extension that has become a fundamental construction for understanding modal logics and their algebraic counterpart.  Our motivation stems from one of the central areas in model theory of modal  logics --  namely its interconnection with first-order logic. This topic has been tackled from various sides. Example of this include: correspondence theory,   \cite{correspondence, Sahlqvist},  definability  \cite{goldblatt, Hollenberg},  bounded fragments of first-order logic  \cite{hajnal}, canonicity \cite{ Benthem, goldblatt1, goldblatt3, fine}. Some of these connections are guided by the ``proportion'':
	\begin{align*}
		\frac{\text{Modal logic}}{\text{Bisimulation}} \approx 	\frac{\text{First-order logic }}{\text{Partial isomorphism }} 
	\end{align*}
	That is  ``\textit{bisimulations are to  modal logic, what partial isomorphism is for first-order logic}''. This was put forward  by de Rijke in \cite{rijke} and was further explicated through the investigation of bounded  fragments.  Later,  in  \cite{Venema} Venema  improved de Rijke's results by invoking the ``proportion'':
	\begin{align*}
		\frac{\text{Modal logic}}{\text{Ultrafilter extension}} \approx 	\frac{\text{First-order logic }}{\text{Ultrapowers}} 
	\end{align*}
	By introducing the construction called \textit{ultrafilter union}, Venema justified  the phrase  \textit{``taking an ultrafilter union of a set of models is the modal analogue of taking ultraproducts in the model theory of first-order logic''}.\footnote{We should mention that all of the upcoming studies can be adapted to ultrafilter unions and ultraproducts.} From these constructions one easily obtains the modal counterpart of the well-known Keisler-Shelah Theorem: Two (pointed) models are modally equivalent if and only if they have bisimilar ultrafilter extensions (cf. \cite{black} Theorem 2.62). 
	
	Our primary focus will be on the second ``proportion'' by studying the two constructions, and investigating how the resulted new structures can be related to each other. From this approach,  the following questions mainly issued by \cite{fine,deltasigma, Benthem, modclass} are very natural:
	\begin{enumerate}
		\item[•] What can be said about the \textit{modal theory} of the ultrapowers?
		
		\item[•] What can we say about the \textit{first-order theory} of the ultrafilter extensions?
		
		\item[•] What kind of \textit{modal formulas} are preserved under taking ultraflter extension?
		
		\item[•] What kind of connections can we have between the  ultrafilter extensions and ultrapowers?
	\end{enumerate}
	
	\noindent In general, it is not surprising that the ultrafilter extension, as opposed to ultrapowers, may fail to preserve first-order formulas.  As a result, the literature mostly posed and partially answered the question;  ``\textit{What type of first-order formulas are preserved under taking ultrafilter extensions?}'' \cite{correspondence, modclass, Benthem, goldblatt}. This turned out to be $\Pi_1^1$-hard, hence no such characterization can exist \cite{cate}.
	Also, similar questions  concerning the first type of ultrafilter extension were addressed in \cite{Saveliev0}  (cf. Problem 5.2, and  Theorem 5.3). Here we take the reverse approach, surprisingly  untouched by the literature and ask; \textit{``What type of structures are elementary equivalent to their ultrafilter extension?''}. 
	Hence, instead of formulas, we concentrate on structures, attempting to isolate classes for which ultrafilter extensions and ultrapowers share structural features.
	Our main results are proved in Section 4.   We show that for certain structures (called bounded graphs) their ultrafilter extension has a characteristic similar to \L o\'s's Lemma. As a result of this, they are found to be elementary substructures to their ultrafilter extension. Additionally,  their modal logics coincide.

	\section{Preliminary notions, background}
	
	\paragraph{Semantics}  Here we recall all the necessary notions from the semantics of modal logics and fix the notation. The reader who is more interested in this area is referred to \cite{Zakarcsibacsi, black, black2} and \cite{chang}. Throughout the paper we consider structures $\mathfrak{F}=\langle W, R\rangle$ to be \textit{directed graphs} for the first-order language $\mathcal{L}_R$ with equality, containing a single binary relation symbol  $R$.
	
	As opposed to the first-order relational language $\mathcal{L}_R$, the standard propositional modal language $\mathcal{L}_\diamondsuit$ is the extension of the propositional language by  a single  unary connective $\diamondsuit$.  A \textit{model} for $\mathcal{L}_\diamondsuit$ is a pair $\M=\langle \F, V\rangle$, where $\F$ is the underlying structure and $V:\Phi\to \mathcal{P}(W)$ is a function, called \textit{valuation}.  The relation $\M,w\Vdash \varphi$ reads as the truth of $\varphi\in  \mathsf{Form}(\mathcal{L}_\diamondsuit)$ in the model $\M$ at $w\in W$, and is defined in the usual recursive manner. The only non-obvious case is
	
	\begin{enumerate}
		\item[•] $\M,w\Vdash \diamondsuit\varphi \Leftrightarrow (\exists v\in W)Rwv$ and $\M,v\Vdash \varphi$.
	\end{enumerate}
	Moreover, $\M\Vdash\varphi$ denotes the fact that for all $w\in W$ we have $\M,w\Vdash \varphi$, and  $\F\Vdash \varphi$ abbreviates the statement that  for all valuation $V$,  the relation $\M\Vdash\varphi$ holds. Finally, the \textit{modal theory} (or \textit{logic}) of $\F$ is denoted by $\Lambda(\F)=\{\varphi\in\mathsf{Form} (\mathcal{L}_\diamondsuit): \F\Vdash\varphi\}$.  Since our motivation is to study the first-order structural and modal properties both appearing at the level of frames,  we confine ourselves into this segment of the semantics.
	\paragraph{Ultrafilter extension}  Let  $\F=\langle W,R\rangle$ be a structure and $X\subseteq W$. We define the following operations on $\mathcal{P}(W)$:
	\begin{align*}
		R^+(X) & = \{s\in W: \exists w(Rws\wedge w\in X)\}\\
		R^-(X) & = \{w\in W: \exists s(Rws\wedge s\in X)\}\\
		R(X) &= R^-(X) \cup R^+(X)\\
		l_R(X) & = \{w\in W : \forall v(Rwv\to v\in X)\}
	\end{align*}
	
	\noindent If $X=\{w\}$, we use the shorthand notation $R(w)$ instead of $R(\{w\})$ and similarly for the other operations. The \textit{ultrafilter} (or \textit{canonical}) \textit{extension} of $\mathfrak{F}$ is the structure $\uf = \langle \Uf(W),R^\ue\rangle$, where $\Uf(W)$ is the set of all ultrafilters over $W$ and the relation $R^\ue$ is defined as:
	\[R^\ue uv \Leftrightarrow (\forall X\in v)(R^-(X)\in u)\]
	It is easy to see that the following equivalence holds between the ultrafilter relation and  the operations defined above:
	\[R^\ue uv \Leftrightarrow\{Y: l_R(Y)\in u\}\subseteq v \Leftrightarrow \{R^+(X): X\in u\}\subseteq v.\]
	By $\pi_w =\{ X\subseteq W: w\in X\}$ we denote the principal ultrafilter generated by $w$, and $\Uf^*(W)$ denotes the set of non-principal ultrafilters. The  function $\eta: \mathfrak{F} \hookrightarrow\uf$, where  $\eta(w) = \pi_w$ is an embedding,  called the \textit{canonical embedding} (of $\mathfrak{F}$ into $\uf$).  For a model $\M=\langle \F, V\rangle$, its  ultrafilter extension is given by $\M^\ue =\langle \uf, V^\ue\rangle$, where $V^\ue(p) = \{u\in\Uf(W): V(p)\in u\}$. A well-known fact is the equivalence (cf. Proposition 2.59 from \cite{black}):
	\begin{align}\label{truth=member}
		\M^\ue, u\Vdash \varphi \Leftrightarrow V(\varphi)\in u,
	\end{align}
	
	\noindent for all $u\in \Uf(W)$ and $\varphi\in\mathsf{Form}(\mathcal{L}_\diamondsuit)$, from which it follows that  $\Lambda(\F)\subseteq \Lambda(\uf)$.

	\paragraph{Model theoretic notions} Using  standard notation, for $\mathfrak{F} =\langle W,R\rangle$, and $\mathfrak{F}' =\langle W',R'\rangle$ by   $\F \leq \F'$, we denote the fact that $\F$ is   a substructure of $\F'$ and  $\F\preceq \F'$ means that $\F$ is an elementary substructure of $\F'$.  Elementary equivalence is denoted by $\F\equiv \F'$.
	We recall that a set $\Gamma$ of first-order formulas in a language $\mathcal{L}$ with a free  variable $x$ is \textit{satisfied} in the structure $\mathfrak{F}=\langle W,R\rangle$, if there is a $w\in W$, such that $\F\models \varphi(w)$, for each $\varphi(x)\in \Gamma$. Then $\Gamma$ is \textit{finitely satisfiable} if each finite subset of $\Gamma$ is satisfied in $\F$. For a cardinal $\kappa$, the structure $\F$ is \textit{$\kappa$-saturated }if for each $A\in [W]^{<\kappa}$ and set of formulas $\Gamma$ in the extended language $\mathcal{L}\cup \{c_a: a\in A\}$ containing a free variable $x$, the set $\Gamma$ is satisfied, whenever $\Gamma$ is finitely satisfied. A well-known fact on saturation is the following theorem (cf. Theorem 6.1.8 in \cite{chang}):

	\begin{fact}
		For any structure $\F$ and cardinal $\kappa$ there exists a $\kappa$-saturated ultrapower of $\F$.
	\end{fact}
	
	\paragraph{Ultrapower vs. ultrafilter extension} Our primary interest is in the potential connections between an ultraproduct of $\F$ and $\uf$. According to Lemma 3.8 from  \cite{Benthem},  $\uf$ is always a $p$-morphic image of a suitably saturated ultrapower of $\F$.  Also it is proved in \cite{modclass} (Theorem 15.11) and later generalized in  \cite{goldblatt0} (Theorem 4.2.5), that a first-order formula is preserved under $p$-morphisms if and only if it is equivalent to an  essentially positive formula (for the definition: \cite{goldblatt0}, Section 4).  Hence, essentially positive formulas are preserved under taking ultrafilter extensions. Another  important $\mathcal{L}_R$-fragment is the class of $r(u)$-formulas  introduced  by \cite{modclass} in Definition 15.17, where $u$ is a fixed variable. This class enjoys the following ``\L o\'s Lemma-like'' characterization: For any structure $\F$, $r(u)$-formula $\varphi(u,x_0,\dots, x_n)$ and $w_0,\dots, w_n\in W$, $v\in \Uf(W)$ we have,
	\begin{align}\label{1}
		\uf\models \varphi(v,\pi_{w_0},\dots,\pi_{w_n}) \text{ iff } \{s\in W: \F\models \varphi(s, w_0,\dots w_n)\}\in v.
	\end{align}
	Hence the existential closure of an $r(u)$-formula is preserved by passing from $\F$ to $\uf$. 
	
	The exact syntactic nature of formulas that are preserved under taking ultrafilter extensions is proved to be $\Pi_1^1$-hard by \cite{cate}. As a result, it is not surprising that elementarily equivalence is by far out of reach, as the  following  textbook example shows. 
	
	\begin{exmp}\label{ex}
		Let $\mathfrak{N}=\langle \mathbb{N}, < \rangle$, its ultrafilter extension can be illustrated as:
		
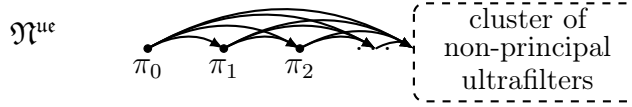
\begin{center}
			\begin{tikzpicture}

				\node [below, black] at (-1.5,-1.5) {$\mathfrak{N}^\ue$};
				
				\draw[fill=black] (0,-2) circle [radius=0.05];
				\node [below, black] at (0,-2) {$\pi_0$};
				
				\draw[fill=black] (1,-2) circle [radius=0.05];
				\node [below, black] at (1,-2) {$\pi_1$};
				
				\draw[fill=black] (2,-2) circle [radius=0.05];
				\node [below, black] at (2,-2) {$\pi_2$};

				\node [black] at (3,-2) {$\dots$};

				\filldraw[fill=none, draw=black, thick, rounded corners, dashed] (3.5,-1.4) rectangle (6.5,-2.7);

				\draw[black,  -latex,thick] (0,-2) to [bend left] (1,-2);
				\draw[black,  -latex,thick] (0,-2) to [bend left] (2,-2);
				\draw[black,  -latex,thick] (0,-2) to [bend left] (3,-2);
				\draw[black,  -latex,thick] (0,-2) to [bend left] (3.5,-2);

				\draw[black,  -latex,thick] (1,-2) to [bend left] (2,-2);
				\draw[black,  -latex,thick] (1,-2) to [bend left] (3,-2);
				\draw[black,  -latex,thick] (1,-2) to [bend left] (3.5,-2);
				
				\draw[black,  -latex,thick] (2,-2) to [bend left] (3,-2);
				\draw[black,  -latex,thick] (2,-2) to [bend left] (3.5,-2);
				\draw[black,  -latex,thick] (3,-2) to [bend left] (3.5,-2);

				\node [black] at (5,-1.6) {{\small cluster of  }};
				\node [black] at (5,-2) {{\small non-principal }};
				\node [black] at (5,-2.4) {{\small  ultrafilters }};

			\end{tikzpicture}

		\end{center}
	\end{exmp}
	
	\noindent where for any ultrafilter $u\in \Uf(\mathbb{N})$ and $v\in \Uf^*(\mathbb{N})$ we have $u <^\ue v$.  Trivially $\mathfrak{N}\models \forall x( x\not < x)$, but $\mathfrak{N}^\ue\models \exists x( x < x)$, hence $\mathfrak{N}\not\equiv \mathfrak{N}^\ue
	$. \\
	
	Intuitively,  the failure was due to the fact that each  $n\in \mathbb{N}$ had infinitely many $<$-successors. In order to characterize subclasses e.g. whose elements are elementary equivalent to their ultrafilter extensions, it is natural to impose conditions on the relation in question. 
	
	\section{General properties on degree}

	In what follows we are going to present some basic properties on the branching of an ultrafilter extension.
	\begin{defn}
		Let $\F =\langle W,R\rangle$ be a structure.  We introduce the notation:
		\begin{enumerate}
			\item[•] $\deg^+(w) = | R^+(w)|$
			\item[•] $\deg^-(w) = |R^-(w)|$
			\item[•]$\deg(w) = \deg^+(w) +\deg^-(w)$
		\end{enumerate}
		Also,  $\F$ is
		\begin{enumerate}
			\item[•] \textit{ $\deg^+$-finite} (\textit{$\deg^-$-finite, $\deg$-finite}),  if for all $w\in W$, $\deg^+(w)<\omega$ (resp. $\deg(w),\deg^-(w)<\omega$).  
			\item[•]   \textit{$\deg^-$-bounded} (\textit{$\deg^+$-bounded, $\deg$-bounded}), if there is some $n\in \omega$, such that for all $w\in W$, $\deg^+(w) \leq n$ (resp. $\deg(w), \deg^-(w) \leq n$). 
		\end{enumerate}
		Finally,  $\F$ is \textit{locally finite} if $\F$ is $\deg$-finite, moreover  $\F$ is  \textit{bounded} if $\F$ is $\deg$-bounded.
	\end{defn}
	
	In basic modal logic,  the term \textit{image finite} is used for  $\deg^+$-finite structures that are of special importance, as they have the Hennessy-Milner property. That is given  image finite models $\M$, $\M'$,  for every  $w\in W$ and $w'\in W'$ we have $\{\varphi\in \mathsf{Form}(\mathcal{L}_\diamondsuit): \M,w\Vdash\varphi\} = \{\varphi\in \mathsf{Form}(\mathcal{L}_\diamondsuit): \M',w'\Vdash\varphi\}$ if and only if the two states $w$ and $w'$ are bisimilar (cf. \cite{black}).  Our main results, however,  will be concentrated on the bounded structures.
	
	\begin{defn}
		Let  $u_0,\dots, u_n$ be  \textit{different} ultrafilters over $W$. We say that the subsets $D_{u_0},\dots, D_{u_n} \subseteq W$ are \textit{distinguishing elements} for the ultrafilters $u_0,\dots, u_n$, if $D_{u_i}\cap D_{u_j} = \emptyset$, whenever $i\neq j \leq n$, and
		$D_{u_i}\in u_j$ iff $i=j$. 
	\end{defn}
	
	\begin{prop}\label{prop0} Fix any structure $\mathfrak{F}=\langle W,R\rangle$ and $u\in \Uf(W)$. If $\{w\in W: \deg^+(w) \leq n \}\in u$, then $\deg^+(u)\leq n$.
	\end{prop}
	\begin{proof}
		\noindent Let $X= \{w\in W: \deg^+(w) \leq n \}$. By way of contradiction, assume that $\deg^+(u)> n$. For $i\leq n$ let $v_i\in \Uf(W)$ be different ultrafilters, such that $R^\ue uv_i$ and fix some distinguishing elements $D_{v_i}\in v_i$.  Then $\bigcap_{i\leq n} R^-(D_{v_i})\cap X\in u$. Hence there is some $s\in X$ with $\deg^+(s)>n$, which is a contradiction. 
	\end{proof}

	\paragraph{Locally finite structures}
	
	Regarding the first constraint on the relation, we  present some  basic connections between the ultrafilter extension and ultrapower of locally finite structures in terms of generated substructures from the standpoint of modal logic.  A structure $\F'=\< W', R'\>$ is called a \textit{generated substructure} of $\F=\<W,R\>$ (notation: $\mathfrak{F'}\trianglelefteq \mathfrak{F}$), if $\F' \leq  \F$ and moreover, for all $w'\in W'$ and $v\in W$, $Rw'v$ implies $v\in W'$.

	\begin{prop}[Theorem 2.5 in \cite{fasi}]\label{8:1} $\mathfrak{F}\unlhd\uf$ if and only if $\mathfrak{F}$ is $\deg^+$-finite.
		\begin{proof}
			($\Leftarrow$) Suppose $R^\ue\pi_w v$. As $\mathfrak{F}$ is $\deg^+$-finite, there is  some $n\in \omega$ such that $R^+(w)=\{w_0,\dots, w_n\}$.   For any $X\in v$, since $w\in R^-(X)$,  there exists an  $i\leq n$, such that $w_i\in X$.  We claim that $v=\pi_{w_i}$, for some $i\leq n$. Otherwise $W\setminus R^+(w) =\bigcap_{i\leq n} (W\setminus \{w_i\})\in v$. However, for no $w_i\in R^+(w)$ we have $w_i\in W\setminus R^+(w) $, which is a contradiction.

			($\Rightarrow$) Suppose that there is some $w\in W$, such that $\deg^+(w) \geq \omega$. Consider a $v\in \Uf^*(W)$ with $R^+(w)\in v$. It is enough to show that $R^\ue \pi_w v$, as in this case $\mathfrak{F}\not\trianglelefteq\uf$.  Pick any $X\in v$, then $R^+(w)\cap X\in v$. Let $s\in R^+(w)\cap X$, then $Rws$ and $w\in R^-(X)$, that is  $R^-(X)\in\pi_w$. 
		\end{proof}
		
	\end{prop}
	
	\noindent  Now we have the following simple characterization via generated substructures.
	
	\begin{thm}\label{8:7} Let $\mathfrak{F}=\langle W, R\rangle$ be a structure, $|I|\geq \aleph_0$ and $D$ be an $\aleph_1$-incomplete ultrafilter over  $I$. The following are equivalent:
		\begin{enumerate}
			\item[(i)] $\mathfrak{F}$ is $\deg^+$-finite,
			\item[(ii)]$\mathfrak{F}\unlhd \uf$,
			\item[(iii)] $\mathfrak{F}\unlhd \, ^I\mathfrak{F}/D$.
		\end{enumerate}
		
		\begin{proof} 
			(i) $\Rightarrow$ (iii) Suppose that  $\deg^+(w)<\aleph_0$ for all $w\in W$.  In order to avoid cumbersome notation let us denote  $^I\mathfrak{F}/D=\langle W^*, R^*\rangle$ and use  $[c_w]_D\in W^*$  to denote the diagonal element for  $w\in W$. Assume, for the sake of contradiction, that $\mathfrak{F}\not\trianglelefteq  \, ^I\mathfrak{F}/D$. Hence, there is some $[c_w]_D,[f]_D\in\,W^*$, such that  $[f]_D\neq [c_v]_D$ for all $v\in W$, but  $R^*[c_w]_D[f]_D$.  Then $\{i\in I:  Rwf(i)\}\in D$ and for all $v\in R^+(w)$  we have that $\{i\in I:  v\neq f(i)\}\in D$, as $[c_v]_D \neq [f]_D$.  Since $R^+(w)$ is finite, $\emptyset= \bigcap_{v\in R^+(w)}\{i\in I:  v\neq f(i)\}\cap\{i\in I:  Rwf(i)\}\in D$, which is a contradiction.
			
			(iii) $\Rightarrow$ (i) Let $\mathfrak{F}\unlhd\, ^I\mathfrak{F}/D$ and again assume, by way of contradiction, that $\deg^+(w)\geq \aleph_0$,  for some $w\in W$. Consider an enumeration of $R^+(w)=\{v_\alpha : \alpha <\beta\}$.  We construct an $[f]_D\in W^*$, such that  $[f]_D\neq [c_v]_D$ for all $v\in W$,  but  $R^*[c_w]_D[f]_D$.   As $D$ is $\aleph_1$-incomplete,  there exists $\{I_n: n\in\omega\}\subseteq D$, such that $I_0=I$, $I_n \supseteq I_{n+1}$, and $\bigcap_{n<\omega} I_n =\emptyset$. Hence, for all $i\in I$ there is an $m(i)\in \omega$, such that $i\in I_{m(i)}$ and for all $n>m(i)$, $i\not\in I_n$. We define $f\in\, ^IW$ as follows: let  $f(i) = v_{m(i)}$, where $v_{m(i)}\in R^+(w)$.  Then for each $v_\alpha\in R^+(w)$, we have $[f]_D\neq [c_{v_\alpha}]_D$ because of the following reason. Pick $v_\alpha\in R^+(w)$.  If $\omega\leq \alpha <\beta$, then $\{i\in I: f(i)=v_\alpha\}=\emptyset$ by definition of $f$. On the other hand, for $\alpha< \omega$ we have two cases. If $v_\alpha\neq v_{m(i)}$ for any $i\in I$, then again by definition of $f$ we obtain  $\{i\in I: f(i)=v_\alpha\}=\emptyset$. Otherwise $v_\alpha = v_{m(i)}$ for some $i\in I$.   Then $\{i\in I: f(i)=v_{m(i)}\}\subseteq I_{m(i)}\setminus I_{m(i)+1}\not\in D$. However, we have $R^*[c_w]_D[f]_D$, as $\{i\in I: Rwf(i) \}= I\in D$, contradicting to $\mathfrak{F}\unlhd  \, ^I\mathfrak{F}_{/D}$.
			
			The proof of (i) $\Leftrightarrow$ (ii) is Proposition \ref{8:1}.
		\end{proof}
	\end{thm}
	\begin{rem}
		We note that  the  proposition above  admits an immediate dual characterization on $\deg^-$ due to Theorem 2 from \cite{Saveliev4}, according to which,
		\[(R^{-1})^\ue =(R^\ue)^{-1}.\]
	\end{rem} 
	However promising the locally finite structures are,  their  first-order model theoretic connection to ultrafilter extensions is doomed to fail at a very first level.\footnote{We are grateful to the referee for pointing out the next example.}  Consider the structure $\mathfrak{N}= \biguplus_{n\in\omega}\mathfrak{N}_n$, the disjoint union of $\mathfrak{N}_n=\langle n, <\rangle$, for all $n\in \omega$:
	
	\begin{center}

		\begin{tikzpicture}
			\filldraw[fill=white, draw=black,, thick,rounded corners] (-1,.5) rectangle (3,-3.5);
			
			\draw[fill=black] (0,0) circle [radius=0.05];
			\node[black, below] at (0,0) {$0$};
			
			\draw[fill=black] (0,-1) circle [radius=0.05];
			\node[black, below] at (0,-1) {$0$};
			
			\draw[fill=black] (1,-1) circle [radius=0.05];
			\node[black, below] at (1,-1) {$1$};
			
			\draw[fill=black] (0,-2) circle [radius=0.05];
			\node[black, below] at (0,-2) {$0$};
			
			\draw[fill=black] (1,-2) circle [radius=0.05];
			\node[black, below] at (1,-2) {$1$};
			
			\draw[fill=black] (2,-2) circle [radius=0.05];
			\node[black, below] at (2,-2) {$2$};
			
			\node[black, above] at (0,-3.2) {$\vdots$};

			\draw[black, thick, ->] (0,-1) to [bend left ] (1,-1);
			\draw[black, thick, ->] (0,-2) to [bend left ] (1,-2);
			\draw[black, thick, ->] (1,-2) to [bend left ] (2,-2);
			\draw[black, thick, ->] (0,-2) to [bend left ] (2,-2);

		\end{tikzpicture}

	\end{center}

	\begin{prop}$\mathfrak{N}^\ue$ has reflexive points, hence $\mathfrak{N}\not\equiv\mathfrak{N}^\ue$.
		\label{prop3.5}
		\begin{proof} 
			Let $F$ be the family of choice functions of the form $f: \omega\to\bigcup_{n\in\omega} \mathfrak{N}_n$, and let $X_f = \mathfrak{N}\setminus \{f(n): n\in \omega\}$. Then the set $u^* =\{X_f : f\in F\}$ has F.I.P., and any ultrafilter extending $u^*$ is reflexive.  
		\end{proof}
	\end{prop}
	\section{Ultrafilter extension of bounded graphs}

	\noindent From now on, by  $\mathfrak{F}=\langle W,R\rangle$ we will always understand a bounded structure/graph.  These structures were intensively studied in modal logics under the title \textit{bounded alternatives} \cite{kracht, bellissima}.  Our final goal  is to prove $\F\preceq\uf$, which will be established via some ultrapower $^I\F/D$ such that $\F\preceq \uf$ and $\uf\preceq\, ^I\F/D$. Before doing so, we motivate the class of bounded structures by the following warming up example.
	\begin{exmp}
		Let $\F =\langle \mathbb{N}, S\rangle$, where the graph of $S$ is the successor operator. There is an ultrapower $^I\F/D$ of $\F$ such that $^I\F/D\cong \uf$.
		\begin{proof}
			It is not hard to calculate from first principles that for each $u\in \Uf^*(\mathbb{N})$ there is exactly one pair $u^-,u^+\in \Uf^*(\mathbb{N})$ such that $u^-S^\ue u$ and $uS^\ue u^+$, hence $S^\ue$ remains injective. Also, one can show that $S^\ue$ remains  cycle-free. By Pospí\v{s}il's Theorem $|\Uf(\mathbb{N})|= 2^{2^{\aleph_0}}$ and the structure $\uf$ consists of $\<\mathbb{N},S\>$ and a disjoint union of $2^{2^{\aleph_0}}$-many copies of $\<\mathbb{Z}, S\>$. Hence considering an index set $I$ of  cardinality ${2^{\aleph_0}}$ and a regular ultrafilter $D$ over $I$, by standard results (cf. Proposition 4.3.7 of \cite{chang}), the ultrapower $^I\F/D$ will be of cardinality $2^{2^{\aleph_0}}$, and so consequently $^I\F/D\cong \uf$.
		\end{proof}
	\end{exmp}
	
	\paragraph{``\L o\'s Lemma-like'' properties}  Let $\mathsf{K}$ be a class of structures.  Starting from the equivalence of (\ref{1}), we say  that a formula with one free variable $\varphi(x)$ in the language $\mathcal{L}_R$ has the \textit{\L o\'s Lemma-like property}  within $\mathsf{K}$, if for all $\F\in \mathsf{K}$ we have,
	\begin{align}
		\uf\models \varphi(u) \Leftrightarrow \{w\in W: \uf\models\varphi(w) \}\in u,
	\end{align}
	\noindent for all $u\in \Uf(W)$. In what follows, we investigate basic \L o\'s Lemma-like properties within the class of bounded structures.

	The next proposition  is a straightforward application of the de Bruijn-Erdős Theorem, however it  serves as a key observation in proving Theorem \ref{thm2}. Although the usual formulation  concerns undirected graphs, it can be adopted to directed ones. If $\< G,E\>$ is  directed, then a $k$-coloring of $G$ is a function $c: G \to \{0,\dots, k-1\}$ such that for all $x\neq y\in G$
	if $\<x,y\>\in E$,  then $c(x)\neq c(y)$, i.e.  reflexive loops do not ruin the coloring. Also, it is easy to see by induction, that whenever a  finite directed graph $\<G,E\>$ has $\deg(G)\leq k$, then it is $k+1$ colorable.

	\begin{prop}\label{prop1}
		Let $\mathfrak{F} =\langle W,R\rangle$ be a bounded graph. Then for all $u\in\Uf(W)$ we have that $\uf\models Ruu$ iff $\{w\in W: \F\models Rww\}\in u$.
	\end{prop}
	\begin{proof}
		($\Leftarrow$)  Holds trivially for any ultrafilter extension.\\
		($\Rightarrow$) Let $X =\{w\in W: \F\not\models Rww\}$, and  assume that $X \in u$. As $\mathfrak{F}$ is bounded, there is some $m\in \omega$, such that $\deg(w) \leq m$, for all $w\in W$.  Hence, for all finite $H\subseteq X$, the graph $\langle H, {R \upharpoonright H}\rangle$ is $m+1$-colorable. Therefore, the graph $\langle  X,  {R \upharpoonright  X}  \rangle$ is $m+1$-colorable too by the de Bruijn-Erdős Theorem. Such an $m+1$ coloring gives an $m+1$ partition of $\bigcup_{i\leq m}C_i =  X$, and say $C_k\in u$.  Then $R^-(C_k)\cap C_k = \emptyset$, hence $\uf\not\models Ruu $.
	\end{proof}
	
	\begin{prop}\label{prop2}
		Let		$\mathfrak{F}=\langle W,R\rangle$ be a bounded structure, then
		\[ \deg^+(u) = n \text{ if and only if } \{w\in W: \deg^+(w) = n \}\in u.\]
	\end{prop}
	\begin{proof}
		($\Leftarrow$) Suppose that $X = \{w\in W: \deg^+(w) = n\}\in u$. Then by Proposition \ref{prop0} we have that $\deg^+(u)\leq n$. Assume, by way of contradiction, that $\deg^+(u)= m$, for some $m < n$ and fix  distinguishing elements for $R^+(u) =\{v_0,\dots, v_m\}$.  Now, we have two cases. Either, \[A = \{w\in X: (\exists z_w\in R^+(w))(z_w\not \in \bigcup_{0\leq i\leq m} D_{v_i})\}\in u\]
		or
		\[B = \{w\in X: |R^+(w)\cap D_{v_i}| > 2\}\in u\]
		for some $0\leq i\leq m$.
		In the first case, it is easy to see that the system
		\[v^* = \{R^+(Y): Y \in u\}\cup \{\{z_w\in W: w\in A\}\}\]
		has F.I.P., and any ultrafilter $v$ extending $v^*$ has the property $Ruv$, and $v\neq v_i$, for all $0\leq i\leq m$. This contradicts  the assumption that $\deg^+(u) = m$.  For the second case, we define the graph $H = \< B, E\>$, where
		\[\{x,y\}\in E \text{ if and only if } (\exists z\in D_{v_i} )(Rxz \wedge Ry z).\]
		As $\mathfrak{F}$ is bounded, there is some $m\in \omega$, such that $\deg_E(w) \leq m$, for all $w\in H$.  Hence, for all finite $G\subseteq B$, the graph $\langle G, {E \upharpoonright G}\rangle$ is $m+1$-colorable. Consequently $H$ is $m+1$ colorable, and there is some $C_k\in u$, as in Proposition \ref{prop1}. Then for each $w,t\in C_k$ we must have $D_{v_i}\cap R^+(w)\cap R^+(t) = \emptyset$. Hence, consider any partition $Z_0\uplus Z_1=R^+(C_k) \cap D_{v_i}$, where each $Z_i$ picks  at least one element from $R^+(w)\cap D_{v_i}$, for each $w\in C_k$.

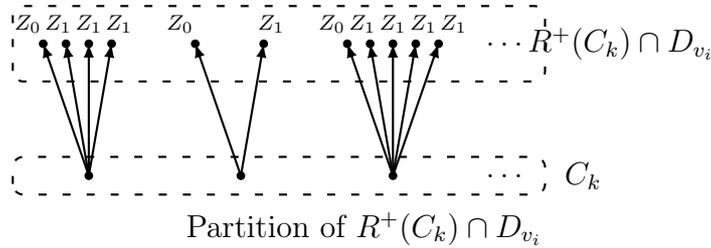
		\begin{center}
			\begin{tikzpicture}
				\filldraw[fill=white, draw=black,loosely dashed, thick,rounded corners] (0,0) rectangle (7,.5);
				
				\filldraw[fill=white, draw=black,loosely dashed, thick,rounded corners] (0,1.5) rectangle (7,2.5);
				\node[black] at (8,2) {$R^+(C_k)\cap D_{v_i}$};

				\node[black] at (7.5,.25) {$C_k$};

				\draw[fill=black] (3,.25) circle [radius=0.05];
				
				\draw[fill=black] (1,.25) circle [radius=0.05];

				\node[black] at (6.5,.25) {$\dots $};
				
				\node[black] at (6.5,2) {$\dots $};

				\draw[fill=black] (5,.25) circle [radius=0.05];

				\draw[black, thick, -latex] (1,.25) -- (1,2);
				\draw[black, thick, -latex] (1,.25) -- (.7,2);
				\draw[black, thick, -latex] (1,.25) -- (1.3,2);
				\draw[black, thick, -latex] (1,.25) -- (.4,2);

				\draw[black, thick, -latex] (3,.25) -- (2.4,2);
				\draw[black, thick, -latex] (3,.25) -- (3.3,2);

				\draw[black, thick, -latex] (5,.25) -- (5,2);
				\draw[black, thick, -latex] (5,.25) -- (4.7,2);
				\draw[black, thick, -latex] (5,.25) -- (5.3,2);
				\draw[black, thick, , -latex] (5,.25) -- (4.4,2);
				\draw[black, thick, -latex] (5,.25) -- (5.6,2);

				\draw[fill=black] (1,2) circle [radius=0.05];
				\node [above, black] at (1,2) {{\tiny $Z_1$}};
				
				\draw[fill=black] (.7,2) circle [radius=0.05];
				\node [above, black] at (.6,2) {{\tiny $Z_1$}};
				
				\draw[fill=black] (1.3,2) circle [radius=0.05];
				\node [above, black] at (1.4,2) {{\tiny $Z_1$}};
				
				\draw[fill=black] (.4,2) circle [radius=0.05];
				\node [above, black] at (.2,2) {{\tiny $Z_0$}};

				\node [above, black] at (2.2,2) {{\tiny $Z_0$}};

				\node [above, black] at (3.4,2) {{\tiny $Z_1$}};
				
				\node [above, black] at (4.2,2) {{\tiny $Z_0$}};
				
				\node [above, black] at (4.6,2) {{\tiny $Z_1$}};
				
				\node [above, black] at (5,2) {{\tiny $Z_1$}};
				
				\node [above, black] at (5.4,2) {{\tiny $Z_1$}};
				
				\node [above, black] at (5.8,2) {{\tiny $Z_1$}};

				\draw[fill=black] (3.3,2) circle [radius=0.05];
				\draw[fill=black] (2.4,2) circle [radius=0.05];
				
				\draw[fill=black] (5,2) circle [radius=0.05];
				\draw[fill=black] (4.7,2) circle [radius=0.05];
				\draw[fill=black] (5.3,2) circle [radius=0.05];
				\draw[fill=black] (4.4,2) circle [radius=0.05];
				\draw[fill=black] (5.6,2) circle [radius=0.05];
				
			\end{tikzpicture}
			
			Partition of $R^+(C_k)\cap D_{v_i}$
			
		\end{center}
		
		Consequently, for $i\in \{0,1\}$ the system  $z^*_i = \{ R[Y] : Y\in u\}\cup\{Z_i\}$ has F.I.P. Therefore each $z^*_i$ extends to some ultrafilter $z_i$ with $R^\ue uz_i$. As $Z_0\uplus Z_1$ is a partition, we may assume that $Z_0\not \in v_i$. Hence $z_0\neq v_i$ and we obtain $\deg^+(u) >m$, which is a contradiction.\\
		($\Rightarrow$) If $\deg^+(u) = n$, then by considering the distinguishing elements for $R^+(u) =\{v_0,\dots, v_n\}$ we obtain $\{w\in W: \deg^+(w) \geq n \}\in u$. If $\{w\in W: \deg^+(w) = n \}\not\in u$, then since $\F$ is bounded, there is some $m > n$, for which $\{w\in W: \deg^+(w) = m \} \in u$. Now we can just simply repeat the whole construction above.
	\end{proof}

	\begin{rem}
		In modal logic, the so called graded modal language ($\mathsf{GML}$) is an extension by $\diamondsuit_k$, for all $k\in \omega$ which has the interpretation,
		\[\M,w\Vdash \diamondsuit_k\varphi \Leftrightarrow |\{v\in W: Rwv \text{ and }\M,v\Vdash\varphi\}| \geq k.\]
		One could think that an equivalence like (\ref{truth=member})  would shrink the proofs above into one single line.  However,  there is no appropriate notion of an ultrafilter extension for $\mathsf{GML}$  in frame semantics  (cf. \cite{sano}).
		Moreover,  the property ``\textit{every element has branching $n$}'' is not expressible by an existential closure of any $r(u)$ formula and is not preserved under $p$-morphisms. Nonetheless, it  is preserved under ultrafilter extensions.  
	\end{rem}

	\begin{defn}
		Let $\F$ be a structure and $w\in W$. For a fixed $n\in \omega$ we define the sets:
		\begin{enumerate}
			\item[•] $\langle w\rangle^0 = \{w\}$,
			\item[•] $\langle w\rangle^{i+1} = R(\langle w\rangle^{i})\cup \langle w\rangle^i$ for $i\leq n$.
		\end{enumerate}
		Then the $n$-\textit{Hull} of $w$ is  the structure $\langle \langle w\rangle^n, R\!\upharpoonright {\langle w\rangle^n}\rangle$. Also,  we blurry the distinction between an $n$-Hull and its underlying set.   We let $\langle w\rangle^\omega =\bigcup_{n\in \omega} \langle w\rangle^n$, and call an element  $s$ an \textit{endpoint} of $ \langle w\rangle^n$, if $s \in \langle w\rangle^n \setminus\langle w\rangle^{n-1}$.
	\end{defn}
	
	\paragraph{Notation} Given  $\< w\>^n$, we let $\varphi_{\langle w\rangle^n}(x)$  be the $\mathcal{L}_R$-formula describing all the possible connections  between the elements of $\langle w\rangle^n$ with respect to $w$. Such a formula always exists, as $\F$ is bounded.  By $\langle w\rangle^n\cong \langle v \rangle^n$, we abbreviate the fact that there is an isomorphism  $f:\langle w\rangle^n \to \langle v\rangle^n$, such that $f(w) = v$. Cleary,   $\langle w\rangle^n\cong \langle v \rangle^n$ if and only if $\F\models \varphi_{\langle w\rangle^n}(v)$. \\
	
	Our next goal is to prove the  following \L o\'s Lemma-like characterization within the class of bounded structures: For every bounded structure $\F$, $n\in \omega$ and $u\in \Uf(W)$ we will show sthat  $\{w\in W: \langle w\rangle^n \cong \langle u\rangle^n\}\in u $.  For this purpose, we are going to use the following example to outline the main ideas.
	
	\begin{exmp}\label{exm0}
		Let $\mathfrak{F}=\langle W,R\rangle$ be a structure $u\in \Uf^*(W)$, and consider its 1-Hull: $\langle u\rangle^1 =\<\{u, v_0,v_1,v_2,v_3\}, R^\ue\upharpoonright\<u\>^1 \> $, where 
		\[R^\ue\upharpoonright\<u\>^1 =\{\<u, v_i\>: i\in \{0,1\}\}\cup \{\<v_i, u\>: i\in \{2,3,4\}\}\cup \{\<v_0,v_2\>\}\]
		that is:
	\end{exmp}
	
	
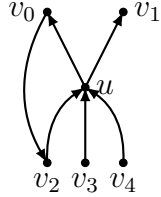
\begin{center}
		\begin{tikzpicture}
			
			decoration = {snake,   
				pre length=3pt,post length=7pt,
			},

			\draw[fill=black] (0, 0) circle [radius=0.05];
			\node [black, right] at (0,0) {$u$};

			\draw[fill=black] (-.5, 1) circle [radius=0.05];
			\node [black, left] at (-.5, 1) {$v_0$};

			\draw[fill=black] (.5, 1) circle [radius=0.05];
			\node [black, right] at (.5, 1) {$v_1$};

			\draw[fill=black] (0, -1) circle [radius=0.05];
			\node [black, below] at (0, -1) {$v_3$};

			\draw[fill=black] (-.5, -1) circle [radius=0.05];
			\node [black, below] at (-.5, -1) {$v_2$};
			
			\draw[fill=black] (.5, -1) circle [radius=0.05];
			\node [black, below] at (.5, -1) {$v_4$};

			\draw[black,  -latex, thick] (0,0) to [] (-.5,1);
			
			\draw[black,  -latex, thick] (0,0) to [] (.5,1);

			\draw[black,  -latex, thick] (0, -1) to [] (0,0);
			
			\draw[black,  -latex, thick] (-.5, -1) to [bend left] (0,0);
			
			\draw[black,  -latex, thick] (.5, -1) to [bend right] (0,0);

			\draw[black,  -latex, thick] (-.5, 1) to [bend right] (-.5, -1);

		\end{tikzpicture}
	\end{center}
	
	\noindent Our aim is to show that $\{w\in W:  \langle u \rangle^1 \cong \langle w\rangle^1\}\in u$.\\
	\begin{proof}
		
		\noindent\textsc{Step 1:} Let $D_u,D_{v_0},\dots, D_{v_4}$ be  pairwise disjoint distinguishing elements for $u,v_0,\dots, v_4$ and consider $U_0 = D_u\cap \{w\in W : \deg^+(w) = \deg^+(u), \deg^-(w) = \deg^-(u)\}$.  By Proposition \ref{prop2} we have $U_0\in u$.  Define the set $U_1$ as:
		\begin{align}
			R^-(D_{v_0})\cap R^-(D_{v_1})\cap R^+(D_{v_2})\cap R^+(D_{v_3})\cap R^+(D_{v_4})\cap U_0.
		\end{align}
		Then $U_1\in u$.   Now,  for each $w\in U_1$  there is exactly one  $s_{w,0}\in R^+(w)$ for which $s_{w,0}\in D_{v_0}$, and exactly one $s_{w,1}\in R^+(w)$ for which $s_{w,1}\in D_{v_1}$. Moreover, $s_{w,0}\neq s_{w,1}$. This  is due to $D_{v_1}\cap D_{v_0} =\emptyset$  and  the construction of $U_1$. The same is true for all the members in $R^-(w)$. Hence, for each $w\in U_1$, we let $f_w: \langle u\rangle^1\to \langle w\rangle^1$ to be the function defined as:
		\begin{align*}
			f_w(u) &= w\\
			f_w(v_i) &= s_{w,i}
		\end{align*}
		We will show that $\{w\in W: f_w: \langle u \rangle^1 \cong \langle w\rangle^1\}\in u$.\\
		
		\noindent \textsc{Step 2:}  First, we claim that
		\begin{align}
			U_2 = \{ w\in U_1: Rwf_w(v_0) \wedge Rf_w(v_0) f_w(v_2) \wedge Rf_w(v_2) w\}\in u.
		\end{align}

		\noindent Suppose otherwise. Then we have:
		\begin{align}
			U^* = \{ w\in U_1:\big(Rf_w(v_2) w\wedge Rw f_w(v_1) \big)\to \neg Rf_w(v_0) f_w(v_2) \}\in u
		\end{align}
		
		\noindent	Next, we define the following sets:
		\begin{align*}
			A &= R^+(U^*)\cap D_{v_0} \in v_0\\
			B &= R^+(A)\cap D_{v_2} \in v_2\\
			C &= R^+ (B) \cap U^*\in u
		\end{align*}
		
		\noindent Pick any $w\in C$.  As $w\in R^+(B)$  and $w\in U^*$, we have that $f_w(v_2)\in D_{v_2}\cap R^+(A)$ with $Rf_w(v_2)w$, and this is the only such element.  Since $f_w(v_2)\in R^+(A)$, we have that $Rtf_w(v_2)$ for some $t\in  R^+(U^*)\cap D_{v_0}$. Observe  that $t\neq f_w(v_0)$, otherwise   $Rwf_w(v_0)$  would contradict to $w\in U^*$. Also $t\in R^+(U^*)$ implies an existence of some $w'\in U^*$, such that $Rw't$. Since $t\in D_{v_0}$, we must have   $w\neq w'$.  Hence, let us define 
		\begin{align}
			\nabla(x) = \{ y\in C:  (\exists s \in B)(\exists t\in A)(Rsx\wedge Rts\wedge Ryt)\}
		\end{align}
		\noindent  and define the graph $\langle C, E \rangle$ such that:
		\begin{align}
			\{x,y\}\in E \iff y\in \nabla(x) \text{ or } x\in \nabla(y)
		\end{align}
		
		\begin{enumerate}
			
			\item[•] As the original structure $\mathfrak{F}$ is bounded, there is some $m\in \omega$, such that for all finite subgraphs $H\leq C$, we have that $\deg_E(H) \leq m$. Hence  $\langle H, E\rangle$ is $m+1$ colorable, and so is $\langle C, E\rangle$ by the de Bruijn-Erdős Theorem via some coloring $c$.\footnote{In general, the exact value of  $m\in\omega$ depends on $\nabla(X)$, and on the original structure $\F$ itself. Nevertheless, there is some $m\in\omega$, such that for each $H\leq C$, $\deg_E(H)\leq m$, since $\F$ is bounded. This  also applies to the corresponding part of Theorem \ref{thm2}.}
			\item[•] This coloring gives an $m+1$ partition of $\bigcup_{\ell \leq m} M_\ell =C\in u$. Let's say $M_k\in u$.
			Then consider the sets;
		\end{enumerate}
		\begin{align*}
			A' &= R^+(M_k)\cap D_{v_0} \in v_0\\
			B' &= R^+(A')\cap D_{v_2} \in v_2\\
			C' &= R^+ (B') \in u
		\end{align*}
		\begin{enumerate}
			\item[] Since $A'\subseteq A$ and $B'\subseteq B$, for all $s\in C'$ there is some $t\in M_k$ such that $sEt$. Hence $c(s)\neq c(t) = k$. However, $M_k\cap C' =\emptyset$,  as every member of $C'$ has different color from $k$, which is a  contradiction.
		\end{enumerate}
		
		\noindent\textsc{Step 3:} Finally, we need to assure the existence of  $V\subseteq U_2$, such that $V\in u$,  and no ``extra'' relations between the element are added.  We can single out these extra relations one by one.  As an 
		example, assume, for the sake of contradiction, that
		\begin{align}
			V^* = \{ w\in U_2: Rwf_w(v_0)\wedge Rwf_w(v_1)\wedge Rf_w(v_0)f_w(v_1) \}\in u.
		\end{align}

		\noindent That is, for all $w\in V^*$, the map $f_w: \<u \>^1\to \< w\>^1$ is not an isomorphism, as  $f_w(v_i) = s_i$ and $Rf_w(v_0)f_w(v_1)$, but not $R^\ue v_0v_1$. Since not $R^\ue v_0v_1$, there is some $X\in v_0$ such that $R^+(X)\not\in v_1$, hence
		\begin{align*}
			A &= X \cap D_{v_0}\in v_0\\
			B &=   \big( W\setminus {R^+(X)} \big)\cap D_{v_1} \in v_1
		\end{align*}
		Then it is easy to see why $\emptyset = R^-(A)\cap R^-(B)\cap V^* \in u$. Otherwise let $w$ be an element in the intersection. Since $w\in R^-(A)$, there is some $s\in  X$, such that $Rws$. As $s\in D_{v_0}$, by the assumption on $w\in V^*$, we have that $s=f_w(v_0)$. Similarly, we get $f_w(v_1)\in B$. However,   $w\in V^*$ implies $Rf_w(v_0)f_w(v_1)$, whereas $f_w(v_0)\in X$ and $f_w(v_i)\in \big( W\setminus {R^+(X)} \big)$ implies  $\neg Rf_w(v_0)f_w(v_1)$. This gives a contradiction.  By iterating this process  we obtain $\{w\in U_2: f_w: \langle w \rangle^1 \cong \langle u\rangle^1\}\in u$. But then $\{w\in U_2: f_w: \langle w \rangle^1 \cong \langle u\rangle^1\} \subseteq\{w\in W:\langle w \rangle^1 \cong \langle u\rangle^1\}\in u$. 
	\end{proof}

	\paragraph{Notation.} Before turning to Theorem \ref{thm2} let us introduce some useful concepts and notation. For a fixed structure $\mathfrak{F}=\langle W,R\rangle$ and $s,t\in W$, we say that there is a \textit{road} from $s$ to $t$ if there are pairwise different elements $w_0,\dots,w_n$ such that $s=w_0$, $t=w_n$ and $w_0R_{1}^\circ\dots R_{n}^\circ w_n$, where each $R_i^\circ\in \{R,R^{-1}\}$.  Then $R[s,t]$  abbreviates the existence of a concrete road from $s$ to $t$. To avoid cumbersome notation, if we already picked a road $R[s,t]$, by $R[t,s]$ we mean its ``inverse'', i.e. the road with $(R_i^\circ)^{-1}$ in each step. When we say that there are $s=w_0,\dots, w_n =t \in W$ such that $R[s,t]$, we mean that the particular road $R[s,t]$ was constructed from  these elements.  Finally, if there are roads $R[w,t]$ and $R[t,s]$, then their ``composition'' from $w$ to $s$ along $R[w,t]$ and $R[t,s]$ is denoted by $ R[w,t,s]$.

	\begin{defn}
		Let $u,v\in \Uf(W)$ and suppose that there are $v_0,\dots, v_n $ with $u=v_0$ and $v=v_n$ such that  $u(R^{\ue})_1^\circ\dots (R^{\ue})^\circ_{n}v $ is a road.  The \textit{ultrafilter road of $R^\ue[u,v]$ based on $X\in u$} is the set $\Delta(X, R^\ue [u,v])$ defined  as:
		
		\begin{align*}
			\Delta_0\big(X, R^\ue[u,v]\big) &= X\\
			\Delta_{i+1}\big(X,R^\ue[u,v] \big) &= \begin{cases}
				R^+\Big(\Delta_i\big(X, R^\ue[u,v] \big) \Big)\cap D_{v_{i+1}} & \text{if } (R^\ue)^\circ_{i+1} = R^\ue\\
				R^-\Big(\Delta_i\big(X, R^\ue[u,v] \big) \Big)\cap D_{v_{i+1}}   & \text{if } (R^\ue)^\circ_{i+1} = (R^\ue)^{-1}
				\vspace{.2cm}
			\end{cases}
		\end{align*}
		
		\noindent	Finally let $\Delta\big(X,R^\ue [u,v]\big) := \Delta_{n}\big(X,R^\ue [u,v]\big)$.\\
	\end{defn}
	\begin{defn} Let  ${R} [w_0,w_n]$ be a particular road along some $w_0,\dots,w_n \in W$ and $R^\ue[v_0,v_m]$ be another road along some $v_0,\dots, v_m\in \Uf(W)$. Then ${R} [w_0,w_n]$  is \textit{ultrafilter similar to} ${R^\ue} [v_0,v_m]$ if $n=m$,  moreover for each $0\leq i\leq n$ the following holds:
		\begin{enumerate}
			\item[•] if  $R^\circ_i = R$, then $(R^\ue)_i^\circ = R^\ue$,
			\item[•] if  $R_i^\circ = R^{-1}$, then $(R^\ue)^\circ_i = (R^\ue)^{-1}$,
			\item[•]$w_i\in D_{v_i}$.
		\end{enumerate}
		
	\end{defn}
	\noindent	Now we are ready to prove our key theorem.
	\begin{thm}\label{thm2}
		Let $\mathfrak{F}$ be a bounded graph. Then for all $n\in\omega$ and $u\in \Uf^*(W)$, we have: 
		\[\{w\in W: \langle w\rangle^n \cong \langle u\rangle^n\}\in u. \] 
	\end{thm}
	\begin{proof}
		\textsc{Step 1:} Pick any $u\in \Uf^*(W)$ and $n\in \omega$.  Let $u,v_0,\dots, v_k \in \langle u\rangle^{n}$ be an enumeration with distinguishing elements $D_u,D_{v_0},\dots,D_{v_k}$.  By Proposition \ref{prop1}, we may assume that for each $i\leq k$ we have that $D_{v_i} \subseteq \{w\in W: \F\models Rww\}$,  whenever $\uf\models Rv_iv_i$ and $D_{v_i} \subseteq \{w\in W: \F\not\models Rww\}$  whenever $\uf\not\models Rv_iv_i$.   
		Suppose for some $m<n$ we already know that there is some $U_m \subseteq D_u$, such that $U_m\in u$ with the following property: For all $w\in U_m$
		
		\begin{enumerate}
			\item[(i)] there is an isomorphism  $f_{w,m}: \langle u\rangle^m\to\langle w\rangle^m$,  such that $f_{w,m}(u)= w$,
			\item[(ii)] $f_{w,m}(v_i)\in D_{v_i}$.
		\end{enumerate}
		
		\noindent 	 As for the base case  $m=0$, we have that $U_0\in u$, again by Proposition \ref{prop1}. In what follows, up to \textsc{Step 4} we present a technique to obtain:  
		\begin{align}
			U_{m+1} = \left\{ w\in U_m\ :  \parbox[c]{3in}{\centering
				$ \exists f_{w,m+1}: \langle w\rangle^{m+1}\cong \langle u\rangle^{m+1}, f_{w,m}\subseteq f_{w, m+1},$\\
				$(\forall v\in \<u\>^{m+1})f_{w,m+1}(v)\in D_v $
			}\right\} \in u.
		\end{align} 
		Let  $e_0,\dots, e_r$ be an enumeration of all the endpoints of $\langle u\rangle^m$, and consider $e_0$.
		\begin{center}
			\begin{tikzpicture}
				
				\filldraw[fill=none, draw=black, thick, rounded corners, dashed] (0,2) rectangle (2,4);
				
				\node [left, black] at (3,4) {$\langle u\rangle^m$};      
				
				\draw[fill=black] (1,3) circle [radius=0.05];
				\node [left, black] at (1,3) {$u$};      
				
				\draw[fill=black] (.5,2.5) circle [radius=0.05];
				\draw[fill=black] (1,2.5) circle [radius=0.05];
				\draw[fill=black] (1.5,2.5) circle [radius=0.05];
				\draw[fill=black] (1,3.5) circle [radius=0.05];
				\draw[fill=black] (1.5,3) circle [radius=0.05];

				\draw[black,  -latex, thick,line join=round,
				decorate, decoration={
					zigzag,
					segment length=4,
					amplitude=.9,post=lineto,
					post length=2pt
				}] (.5,2.5) to [] (1,3);
				
				\draw[black,  -latex, thick,line join=round,
				decorate, decoration={
					zigzag,
					segment length=4,
					amplitude=.9,post=lineto,
					post length=2pt
				}] (1,2.5) to [] (1,3);
				
				\draw[black,  -latex, thick,line join=round,
				decorate, decoration={
					zigzag,
					segment length=4,
					amplitude=.9,post=lineto,
					post length=2pt
				}] (1.5,2.5) to [] (1,3);
				
				\draw[black,  -latex, thick,line join=round,
				decorate, decoration={
					zigzag,
					segment length=4,
					amplitude=.9,post=lineto,
					post length=2pt
				}] (1,3) to [] (1,3.5);
				
				\draw[black,  -latex, thick,line join=round,
				decorate, decoration={
					zigzag,
					segment length=4,
					amplitude=.9,post=lineto,
					post length=2pt
				}] (1.5,3) to [] (1,3.5);

				\filldraw[fill=none, draw=gray, thick, rounded corners, dashed] (-1,1.4) rectangle (3,4.7);
				
				\node [left, black] at (4.3,5) {$\langle u \rangle^{n}$};      
				
				\draw[fill=gray,] (-.5,1.6) circle [radius=0.05];
				\draw[fill=gray] (-.5,2.5) circle [radius=0.05];
				\draw[fill=gray] (2.5,3.5) circle [radius=0.05];
				\draw[fill=gray] (2.5,2.5) circle [radius=0.05];
				\draw[fill=gray] (1,4.5) circle [radius=0.05];

				\draw[gray,  -latex, thick,line join=round,
				decorate, decoration={
					zigzag,
					segment length=4,
					amplitude=.9,post=lineto,
					post length=2pt
				},] (-.5,1.6) to [] (.5,2.5);
				\draw[gray,  -latex, thick,line join=round,
				decorate, decoration={
					zigzag,
					segment length=4,
					amplitude=.9,post=lineto,
					post length=2pt
				}] (-.5,2.5) to [] (.5,2.5);
				
				\draw[gray,  -latex, thick,line join=round,
				decorate, decoration={
					zigzag,
					segment length=4,
					amplitude=.9,post=lineto,
					post length=2pt
				}] (1,3.5) to [] (1,4.5);
				\draw[gray,  -latex, thick,line join=round,
				decorate, decoration={
					zigzag,
					segment length=4,
					amplitude=.9,post=lineto,
					post length=2pt
				}] (2.5,3.5) to [] (1.5, 3);
				\draw[gray,  -latex, thick,line join=round,
				decorate, decoration={
					zigzag,
					segment length=4,
					amplitude=.9,post=lineto,
					post length=2pt
				}] (2.5,2.5) to [] (1.5,3);

				\filldraw[fill=white, draw=black,, thick,rounded corners] (-4,1) rectangle (6,-1);
				
				\draw[gray,  -latex, thick] (1,3) to [bend right] (1,0);
				\draw[gray,  -latex, thick] (.5,2.5) to [bend right] (.5,-.5);
				\draw[gray,  -latex, thick] (1,3.5) to [bend left] (1,.5);
				\draw[gray,  -latex, thick] (1,3) to [bend right] (1,0);
				\draw[gray,  -latex, thick] (1.5,2.5) to [bend left] (1.5,-.5);
				\draw[gray,  -latex, thick] (1.5,3) to [bend left] (1.5,0);
				\draw[gray,  -latex, thick] (1,2.5) to [bend right] (1,-.5);

				\draw[fill=black] (1,0) circle [radius=0.05];
				\node [left, black] at (1,0) {$w$};     
				
				\draw[fill=black] (.5,-.5) circle [radius=0.05];
				\draw[fill=black] (1,-.5) circle [radius=0.05];
				\draw[fill=black] (1.5,-.5) circle [radius=0.05];
				\draw[fill=black] (1,.5) circle [radius=0.05];
				\draw[fill=black] (1.5,0) circle [radius=0.05];
				
				\draw[black,  -latex, thick,line join=round,
				decorate, decoration={
					zigzag,
					segment length=4,
					amplitude=.9,post=lineto,
					post length=2pt
				}] (1,0) to [] (1,.5);
				
				\draw[black,  -latex, thick,line join=round,
				decorate, decoration={
					zigzag,
					segment length=4,
					amplitude=.9,post=lineto,
					post length=2pt
				}] (1,-.5) to [] (1,0);
				
				\draw[black,  -latex, thick,line join=round,
				decorate, decoration={
					zigzag,
					segment length=4,
					amplitude=.9,post=lineto,
					post length=2pt
				}] (1.5,-.5) to [] (1,0);
				
				\draw[black,  -latex, thick,line join=round,
				decorate, decoration={
					zigzag,
					segment length=4,
					amplitude=.9,post=lineto,
					post length=2pt
				}] (1.5,0) to [] (1,.5);
				
				\draw[black,  -latex, thick,line join=round,
				decorate, decoration={
					zigzag,
					segment length=4,
					amplitude=.9,post=lineto,
					post length=2pt
				}] (.5,-.5) to [] (1,0);

				\draw[fill=black] (-3,0) circle [radius=0.05];
				\node [left, black] at (-3,0) {$w'$};     
				
				\draw[fill=black] (-3.5,-.5) circle [radius=0.05];
				\draw[fill=black] (-3,-.5) circle [radius=0.05];
				\draw[fill=black] (-3,.5) circle [radius=0.05];
				\draw[fill=black] (-2.5,0) circle [radius=0.05];
				\draw[fill=black] (-2.5,-.5) circle [radius=0.05];
				
				\draw[black,  -latex, thick,line join=round,
				decorate, decoration={
					zigzag,
					segment length=4,
					amplitude=.9,post=lineto,
					post length=2pt
				}] (-2.5,-.5) to [] (-3,0);
				
				\draw[black,  -latex, thick,line join=round,
				decorate, decoration={
					zigzag,
					segment length=4,
					amplitude=.9,post=lineto,
					post length=2pt
				}] (-2.5,0) to [] (-3,.5);
				
				\draw[black,  -latex, thick,line join=round,
				decorate, decoration={
					zigzag,
					segment length=4,
					amplitude=.9,post=lineto,
					post length=2pt
				}] (-3.5,-.5) to [] (-3,0);
				
				\draw[black,  -latex, thick,line join=round,
				decorate, decoration={
					zigzag,
					segment length=4,
					amplitude=.9,post=lineto,
					post length=2pt
				}] (-3,-.5) to [] (-3,0);
				
				\draw[black,  -latex, thick,line join=round,
				decorate, decoration={
					zigzag,
					segment length=4,
					amplitude=.9,post=lineto,
					post length=2pt
				}] (-3,0 ) to [] (-3,.5);

				\draw[fill=black] (5,0) circle [radius=0.05];
				\node [left, black] at (5,0) {$w''$};     
				
				\draw[fill=black] (5,-.5) circle [radius=0.05];
				\draw[fill=black] (4.5,-.5) circle [radius=0.05];
				\draw[fill=black] (5.5,-.5) circle [radius=0.05];
				\draw[fill=black] (5,.5) circle [radius=0.05];
				\draw[fill=black] (5.5,0) circle [radius=0.05];
				
				\draw[black,  -latex, thick,line join=round,
				decorate, decoration={
					zigzag,
					segment length=4,
					amplitude=.9,post=lineto,
					post length=2pt
				}] (4.5,-.5) to [] (5,0);
				
				\draw[black,  -latex, thick,line join=round,
				decorate, decoration={
					zigzag,
					segment length=4,
					amplitude=.9,post=lineto,
					post length=2pt
				}] (5,-.5) to [] (5,0);
				
				\draw[black,  -latex, thick,line join=round,
				decorate, decoration={
					zigzag,
					segment length=4,
					amplitude=.9,post=lineto,
					post length=2pt
				}] (5.5,-.5) to [] (5,0);
				
				\draw[black,  -latex, thick,line join=round,
				decorate, decoration={
					zigzag,
					segment length=4,
					amplitude=.9,post=lineto,
					post length=2pt
				}] (5,0) to [] (5,.5);
				
				\draw[black,  -latex, thick,line join=round,
				decorate, decoration={
					zigzag,
					segment length=4,
					amplitude=.9,post=lineto,
					post length=2pt
				}] (5.5,0 ) to [] (5,.5);
				
			\end{tikzpicture}
			
			Illustration of the isomorphism $f_{w,m}: \langle u\rangle^m\to \langle w\rangle^m$
		\end{center}
		We can assume  that $U_m$ has the following additional property: For all $w\in U_m$, $\deg^+(f_{w,m}(e_0)) = \deg^+(e_0)$ and $\deg^-(f_{w,m}(e_0)) = \deg^-(e_0)$. Otherwise let's say that $X =\{w\in U_m: \deg^+(f_{w,m}(e_0)) =k\neq \deg^+(e_0) \}\in u$. Since there is some road $R^\ue[u,e_0]$ in $\<u\>^m$, consider its ultrafilter road $\Delta\big(X, R^\ue[u,e_0] \big)\in e_0$. Observe that if $t\in \Delta \big(X, R^\ue[u,e_0]\big)$, then $t=f_{w,m}(e_0)$ for some $w\in X$, by the assumption on $U_m$.
		Therefore $\emptyset = \Delta\big(X, R^\ue[u,e_0] \big) \cap \{w\in W: \deg^+(w) =\deg^+(e_0)\}\in e_0$ by Proposition \ref{prop2}.  Let $z_0,\dots, z_h$ be an enumeration of $ R^\ue(e_0)\setminus \langle u\rangle^m$, then 
		
		\[U'_m =  \bigcap_{\ell \leq h} \Delta\big(D_{z_\ell} , R^\ue[z_\ell, e_0, u] \big) \cap U_m\in u.\]

		Hence, for  each $w\in U'_m$ let $s_{w,z_\ell}$  denote the unique element in $ R(f_{w,m}(e_0))\setminus \langle w\rangle^m$, for which $s_{w,z_\ell}\in D_{z_\ell}$, moreover $s_{w,z_\ell}$ and $ f_{w,m}(e_0)$ are related by $R$ exactly as $e_0$ and $z_\ell$  are related by $R^\ue$. This fact is abbreviated by  $R^\circ f_{w,m}(e_0) s_{w,z_\ell}$ and  $(R^\ue)^\circ e_0 z_\ell$.\\
		
		\noindent\textsc{Step 2:} Consider $z_0\in R^\ue(e_0)\setminus \langle u\rangle^m$, the first item from the enumeration. Suppose that $(R^\ue)^\circ z_0 v_i$ for some $v_i\in \<u\>^m$.  Then there are some $v_0=u,\dots, v_p = z_0$, such that there is  a road ${R^\ue}[u, e_0, z_0]$ in $\<u\>^m$. Also, there are some $v'_0=v_i,\dots, v'_q = u$, such that there is a
		road ${R^\ue}[v_i, u]$  in $\langle u \rangle^m$. We claim that,
		\begin{align}
			U_{z_0}^{v_i} = \left\{ w\in U'_m\ :  \parbox[c]{2.5in}{\centering
				$\Big( {R}[w, f_{w,m}(e_0), s_{w,z_0} ] \wedge $ \\
				$	\wedge s_{w,z_0}R^\circ f_{w,m}(v_i) \wedge R[f_{w,m}(v_i), w]\Big) $
			}\right\} \in u,
		\end{align}

		\noindent	where  ${R}[w, f_{w,m}(e_0), s_{w,z_0} ] $  is ultrafilter similar to $R^\ue[u,e,v_0]$ and  $R[f_{w,m}(v_i), w]$ is ultrafilter similar to  ${R^\ue}[v_i, u]$.    Otherwise, we have, 
		
		\begin{align}
			U^* = \left\{ w\in U'_m\ :  \parbox[c]{3in}{\centering
				$\Big( R[w, f_{w,m}(v_i)]
				\wedge R[w, f_{w,m}(e_0), s_{w,z_0}] \wedge$ \\	 $\wedge  \neg f_{w,m}(v_i)	R^\circ	s\Big) $
			}\right\} \in u.
		\end{align}
		
		\noindent We define the ultrafilter roads:
		\begin{align*}
			A_\ell &=\Delta_\ell \big(U^*, R^\ue[u,e_0,z_0]  \big)\cap D_{v_\ell} &\text{ for $1 \leq\ell \leq p$}\\
			B &= R^\circ (A_p)\cap D_{v_i}\\
			C_j &= \Delta_j\big(B, R^\ue[v_i,u] \big)\cap U^*& \text{ for $1\leq j\leq q$} 
		\end{align*}

		Let us abbreviate $C = C_q$ and pick $w\in C$.  By the assumption on $U'_m$, and hence on $U^*$, there is exactly one road $R[w, f_{w,m}(v_i)]$ ultrafilter similar to $R^\ue[u,v_i]$.  Also $f_{w,m}(v_i)\in B$. Then $R^\circ f_{w,m}(v_i)s$, for some $s\in \Delta\big(U^*, R^\ue[u,e_0,z_0]\big) $. Then $s\neq s_{w,z_0}$ by assumption and there is a $w'\in U^*$ with a road $R[w',f_{w',m}(e_0),s]$ being ultrafilter similar to $R^\ue[u,e_0,v_0].$ Since $w'\in U^*$, we get $w\neq w'$.  Using the same argument of \textsc{Step 2} from Example \ref{exm0}, let us define;
		\\

		\begin{align}
			\nabla(x) = \left\{ y\in C:  \parbox[c]{3.5in}{\centering $	\Big( \bigwedge_{1\leq \ell \leq p} \exists s_\ell\in \Delta_\ell(U^*, R^\ue [u,e_0,z_0])      \Big)$\\
				$\Big( \bigwedge_{1\leq j\leq q} \exists t_j\in \Delta_j(B, R^\ue [v_i,u])      \Big)	$\\
				$ \Big( {R}[x, s_{p-1}, s_{p} ] \wedge  t_{1}R^\circ s_p  \wedge R[y, t_1]\Big) $
			}
			\right\}
		\end{align}

		\noindent where  ${R}[x,s_{p-1},s_p]$ is composed by the elements $s_\ell$'s ($1\leq \ell \leq p$) and ultrafilter similar to ${R^\ue}[u,e_0,v_0] $ and ${R}[y, t_0]$ composed by the $t_j$'s ($1\leq j\leq q$) and ultrafilter similar to $R^\ue[u, v_i]$.  Define the graph $\langle C, E \rangle$ such that:
		\begin{align}
			\{x,y\}\in E \iff y\in \nabla(x) \text{ or } x\in \nabla(y).
		\end{align}
		
		As $\mathfrak{F}$ is bounded, there is some $m\in \omega$ such that for all finite subgraphs $H$ of $C$ we have that $\deg_E(H) \leq m$. Hence $\langle H, E\upharpoonright H \rangle$ is $m+1$ colorable  and so is $\langle C, E\rangle$ by the de Bruijn-Erdős Theorem.	The coloring gives an $m+1$ partition of $\bigcup_{\ell \leq m} M_\ell =C\in u$. Let's say $M_k\in u$.
		Now consider the sets:
		\begin{align*}
			A'_\ell &=\Delta_\ell \big(M_k, R^\ue[u,e_0,z_0]  \big)\cap D_{v_\ell} &\text{ for $1\leq \ell \leq p$}\\
			B' &= R^\circ (A_p')\cap D_{v_i}\\
			C'_j &= \Delta_j\big(B', R^\ue[v_i,u] \big)\cap U^*& \text{ for $1\leq j\leq q$} 
		\end{align*}
		As above, we use the abbreviation $C'=C'_q$. Since each $A'_\ell\subseteq A_\ell$ and $C'_j\subseteq C_j$, moreover $B'\subseteq B$,  for all $s\in C'$ there is some $t\in M_k$, such that $sEt$. Hence $c(s)\neq c(t) = k$. However, $M_k\cap C' =\emptyset$,  as every member of $C'$ has different color from $k$. This is a  contradiction, since $M_k\cap C'\in u$ by construction. \\

		\noindent\textsc{Step 3:} If $v_{\ell_1}, \dots, v_{\ell_n}$ is an enumeration of $R^\ue(z_0)\cap\langle u\rangle^m$ and  $U_{z_0}^{v_{\ell_{i}}}$ has been  already constructed,  repeat   \textsc{Step 2} by using $U_{v_0}^{v_{\ell_{i}}}$ as the initial set to obtain $U_{v_0}^{v_{\ell_{i+1}}} \subseteq  U_{v_0}^{v_{\ell_{i}}}$, with $U_{v_0}^{v_{\ell_{i+1}}}\in u$. 
		Finally,  for each $w\in U_{z_0}^{v_{\ell_n}} $ put 
		$f '_{w,m}: = f_{w,m}\cup \{\langle z_0,  s_{w, z_0}\rangle\}$. Similarly to  \textsc{Step 3} of Example \ref{exm0}, we can obtain $V_{z_0}\subseteq U_{z_0}^{v_{\ell_{n}}}$, such that $f '_{w,m}: \< u\>^m\cup \{z_0\}\to \< w\>^m\cup \{s_{w, z_0}\}$ is an isomorphism for all $w\in V_{z_0}$.\\

		\noindent \textsc{Step 4:} We can further iterate the process described in \textsc{Step 3}. Consider the next element in $R^\ue(e_0) =\{z_0,\dots, z_h\}$. From  the previously obtained set we get $V_{z_1}\subseteq V_{z_0}$, where $V_{z_1}\in u$ and for all $w\in V_{z_1}$ the function $f'_{w,m}\cup \{\<z_1,s_{w,z_1}\>\}$ is an isomorphism. Hence, after $h$-step we reach $V_{z_h}\in u$.  Then we may repeat all these construction with all the other  endpoints $e_1,\dots, e_r$ from $\langle u\rangle^m$ in order to obtain $U_{m+1}\in u$ and isomorphisms $f_{w,m+1}:\langle u\rangle^{m+1}\to \langle w \rangle^{m+1}$,  for each $w\in U_{m+1}$. By construction it will always be the case that $f_{w,m+1}(v_i)\in D_{v_i}$.\\

		\noindent \textsc{Step 5:} After constructing $U_{m+1}$, we can simply repeat \textsc{Step 1 -- Step 4} to construct $U_{n}\in u$ by using  $U_{m+i}$ as the initial set from which we obtain $U_{m+i+1}$. Finally, we conclude that $ \{w\in W: \langle w\rangle^n \cong \langle u\rangle^n\}\in u $, as $U_n \subseteq \{w\in W: \langle w\rangle^n \cong \langle u\rangle^n\}$ and $U_n\in u$.
	\end{proof}
	Now we derive some important consequences of Theorem \ref{thm2} regarding modal and first-order model theoretic properties. As a direct corollary we have the following.
	\begin{thm} For any bounded graph $\F$ and ultrapower $^I\F/D$, $\Lambda(\F) =\Lambda(\uf)=\Lambda(^I\F/D)$.
		\begin{proof}
			The $\Lambda(\uf) \subseteq \Lambda(\F)$ part always holds as was mentioned earlier.  For the other direction we refer to the notion of  modal depth and $n$-bisimulation from \cite{black}.  Let $\uf\not\Vdash\varphi$ and  $\text{dp}(\varphi)=n$ be the modal depth of $\varphi$.  Then for some valuation $V$ over $\Uf(W)$ and $u\in \Uf(W)$, we have that $\langle\uf,V\rangle, u\not\Vdash\varphi$.  By Theorem \ref{thm2} there is some $w\in W$, such  that $\langle w \rangle^n\cong\langle u\rangle^n$.   Now, let  $U$ be a valuation over $\F$, such that for all  $s\in\langle w\rangle^n$ and $v\in \langle u\rangle^n$ if $s\mapsto v$, then
			\[s\in U(p)\Leftrightarrow v\in V(p).\]
			From the isomorphism of the two $n$-Hulls  it is easy to construct an  $n$-bisimulation between $s$ and $v$.  As a consequence of Proposition 2.31 from \cite{black}, we get that $\langle\F,U\rangle, s\not\Vdash\varphi$ and hence $\Lambda(\F) =\Lambda(\uf)$.
			
			Since $\Sigma_1^1$-formulas are preserved under ultraproducts (cf. \cite{chang}), by the standard translation of modal formulas into the $\Pi_1^1$-fragment,  $\Lambda(^I\F/D)\subseteq \Lambda(\F)$ holds for any structure.  For the other direction a similar $n$-bisimulation argument applies with the aid of  \L o\'s's Lemma.
		\end{proof}
	\end{thm}

	\noindent	The next lemma is  needed to prove Theorem \ref{thm4}.
	
	\begin{lem}\label{lem1}
		Let $^I\F/D$ be an ultrapower of $\F$ over an $\aleph_1$-incomplete ultrafilter $D$. Then for each non-diagonal element $[a]_D\in\,^IW/D$, there is an $X\subseteq \Uf(W)$  with $|X| = {{\aleph_0}}$,  such that for each $u,v\in X$  it holds that $\<v\>^\omega\cong\langle [a]_D\rangle^\omega \cong \langle u\rangle^\omega$ and $ \<u\rangle^\omega\cap\<v\rangle^\omega=\emptyset$.

		\begin{proof}

			Fix some $[a]_D$. By \L o\'s's Lemma for each $n\in \omega$ we get $A_n =\{i\in I: \F\models\varphi_{\<[a]_D\>^n}(a(i))\}\in D$.   Let $A_n^\partial = \{a(i): i\in A_n\}$. Since $[a]_D$ is non-diagonal, each $A_n^\partial$ is infinite.  Therefore, we can	
			consider an injective choice function $f:\omega\to \bigcup_{n\in\omega }A_n^\partial$, that is $f(n)\in A_n^\partial $ and  for all $n\neq m$ we have $f(n)\neq f(m)$. Hence $Y=\{f(n): n\in \omega\}$ is infinite and by Pospí\v{s}il's Theorem there are $2^{2^{\aleph_0}}$ ultrafilters over $Y.$ For each $v^*\in  \Uf(Y)$ fix  one $v\in \Uf(W) $ extending $v^*$, i.e. $v^*\subseteq v$.
			By Theorem \ref{thm2} we obtain
			\[\uf\models \varphi_{\<[a]_D\>^n}(v) \text{ and }  \uf\models \varphi_{\<v\>^n}(v),\]
			
			\noindent 	for all $n\in \omega$. Now, for such a fixed $v\in \Uf(W)$ consider the set $p_{v} =\{\varphi_{\<v\>^n}(x): n\in \omega\}$. Then it is clear  that $^I\F/D\models p_{v}([a]_D)$.  It should be routine to verify  that  in fact $\<[a]_D\>^\omega\cong \< v\>^\omega$, using the assumption of $D$ being $\aleph_1$-incomplete, however  we give a short sketch. 
			
			Since $D$ is $\aleph_1$-incomplete there is a sequence $\< I'_n : n\in \omega\>$ such that $I'_n\supseteq I'_{n+1}$ with $I'_n\in D$ and $\bigcap_{n<\omega}I'_n =\emptyset$.  Consider the decreasing sequence
			\begin{align*}
				I_0& = I'_0,\\
				I_{n+1} &= I'_n \cap A_n.
			\end{align*}
			Then,  for all $i\in I$, there is an $m(i)$ such that $i\in I_{m(i)}\setminus I_{m(i)+1}$. 	For each $n\in \omega$ we let $J_n = \{i\in I: m(i)= n\}$.  Now, since for all $j\in J_n$ the set $\<a(j)\>^{n}$ is finite, say has $n_J\in \omega$ elements, we can construct choice functions $a_1, \dots, a_{n_J-1}: J_{n} \to W$ such that for $j\in J_n$ the set $A_{n,j} =\{a(j), a_1(j),\dots, a_{n_J-1}(j)\}$ has the property \[ \exists f_{n,j} : \<u\>^{n}\cong \< A_{n,j},\  R\upharpoonright A_{n,j}\>. \]
			Pick $z\in \< v\>^\omega$, we  define  the corresponding choice function $a_z: I\to W$ as follows:
		
			\begin{align*}
				a_z(i)=
				\begin{cases}
					f_{m(i),i}(z) &\text{ if }   z\in \<v\>^{m(i)} \\
					b & \text{ otherwise} 
				\end{cases}
			\end{align*}
			where $b$ is arbitrary, but fixed. Then the map $\eta:\<[a]_D\>^\omega \to \<u\>^\omega$ defined by $\eta([a_z]_D)= z$ is an isomorphism.

			Since $|\Uf(Y)| = 2^{2^{\aleph_0}}$, and  for each $v^*\in \Uf(Y)$ we have that $|\<v\>^\omega| \leq \aleph_0$,  by a simple   cardinality argument there must be a subset  $X\subseteq \Uf(W)$ satisfying the conditions of the present lemma.
		\end{proof}

	\end{lem}
	
	We recall the  basic model theoretic facts that if $\mathfrak{A}\leq \mathfrak{B}$, and for each finite sequence $a_0,\dots, a_n\in A$ and $b\in B$ there is an automorphism $\eta: \mathfrak{B}\to \mathfrak{B}$ fixing $\overline{a}$ and moving $b$ into $A$ (i.e.  $\eta(a_i)=a_i$ and $\eta(b)\in A$), then $\mathfrak{A}\preceq\mathfrak{B}$.
	Also, if $\mathfrak{A}\preceq \mathfrak{C}$,  $\mathfrak{A}\leq \mathfrak{B}$ and $\mathfrak{B}\preceq \mathfrak{C}$, then $\mathfrak{A}\preceq\mathfrak{B}$.
	\begin{thm}\label{thm4}
		For any bounded graph $\F$, there exists an ultrapower $^I\F/D$ such that $\F\preceq\uf\preceq\,^I\F/D$.
	\end{thm}
	\begin{proof}
		Let $\F=\langle W,R\rangle$ be any bounded graph with $\lambda = |W|$, and set $\kappa =  2^{2^\lambda} = |\Uf(W)|$. Let $^I\F/D= \langle W^*, R^*\rangle$ be any $\kappa^+$-saturated ultrapower of $\F$ over an $\aleph_1$-incomplete ultrafilter $D$. We claim that $\uf\preceq\,^I\F/D$. First, we show the embedding $\uf \leq\,^I\F/D$. 
		As a consequence of Theorem \ref{8:7} and its dual, for each $w\in W$ we have: 
		\[\langle w\rangle^\omega\cong\langle \pi_w\rangle^\omega\cong \langle [c_w]_D\rangle^\omega.\]
		Hence, we let  $f_0:\uf\to\, ^I\F/D$ to be the partial isomorphism,  where $f_0(\pi_w) = [c_w]_D$. Consider an enumeration $\{u_\gamma\in \Uf^*(W): \gamma <\kappa\}$.
		For $\gamma <\kappa$ we construct  a sequence of  partial isomorphisms $f_\gamma : \uf \to\, ^I\F/D$, such that:
		
		\begin{enumerate}
			\item[a)] $f_\alpha \subseteq f_{\gamma}$, for $\alpha < \gamma$,
			\item[b)] $\langle u_\gamma\rangle^\omega\subseteq \mathrm{dom}(f_{\gamma+1})$, 
			\item[c)] $|\mathrm{dom}(f_\gamma)|  \leq |\gamma| +\aleph_0$,
			\item[d)] $\langle u_\gamma\rangle^\omega \cong \langle f_{\gamma+1}(u_\gamma)\rangle^\omega$,
			\item[e)] $f_\gamma = \bigcup_{\alpha<\gamma}f_\alpha$ if $\gamma$  is limit.
		\end{enumerate}
		Assume that, for some $\beta<\kappa$ we have already defined the  map $f_\beta$, so let us now define $f_{\beta+1}$.  There are two cases: If $u_\beta \in \mathrm{dom}(f_\alpha)$ for some $\alpha < \beta$, then let $f_{\beta +1} = f_\beta$. Otherwise, $u_\beta \not\in \bigcup_{\alpha<\beta}\dom(f_\alpha)$. Observe that in this situation  no $u_\alpha\in\langle u_\beta\rangle^\omega$ is mapped by any previously defined $f_\alpha$ to $^I\F/D$.  
		
		Consider the set $X =\bigcup_{\alpha<\beta} \<f_{\alpha+1}(u_\alpha)\>^\omega$.  We  define the type:
		\[p_{u_\beta}(x)
		= \{\varphi_{\langle u_\beta\rangle^n}(x): n\in\omega\} \cup \{x\neq c: c\in X \} \]
		
		\noindent with parameters from $X$. Since $|X| <\kappa^+$, we have that $|p_{u_\beta}| <\kappa^+$ and since $^I\F/D$ is $\kappa^+$-saturated, we can find an element  $[a]_D\in \,W^*$, such that $^I\F/D\models p_u([a]_D)$. Using the same argument from Lemma \ref{lem1} we can conclude that $\<[a]_D\>^\omega\cong \< u\>^\omega$ and $\<[a]_D\>^\omega\cap X=\emptyset$.  Let us denote this isomorphism by $\eta$ and define $f_{\beta+1}$ as
		
		\[f_{\beta+1}= \bigcup_{\alpha\leq \beta} f_\beta \cup \{\<u,\eta(u)\>: u\in \<u_\beta\>^\omega  \}.\] 
		It is easy to see that $f_{\beta+1}$ indeed satisfies the conditions a)-e). Finally, we set $f = \bigcup_{\gamma <\kappa}f_\gamma$. It is almost trivial that

		\begin{enumerate}
			\item[•] $\mathrm{dom}(f) = \Uf(W)$,
			\item[•] $f$ is an embedding. 
		\end{enumerate}
		Now, we show that $\F^\ue\preceq\, ^I\F/D$. Pick $f(u_0),\dots, f(u_n)\in W^*$ and $[a]_D\in W^*$.  By Lemma \ref{lem1}, there will be $v_0,\dots, v_{n+1}\in \Uf(W)$ such that for $i\leq n+1$ it holds that  $\langle f(v_i)\rangle^\omega\cong \langle [a]_D\rangle^\omega$  and  $\langle f(v_i)\rangle^\omega\cap \langle f(v_\ell)\rangle^\omega=\emptyset$. Hence, for  at least one such $v_i$ we must have that  $\langle f(v_i)\rangle^\omega\cap \langle f(v_j)\rangle^\omega=\emptyset$, for  $i\neq j\leq n$. Then  let $\eta$  be the automorphism that isomorphically permutes each element of $\langle [a]_D\rangle^\omega$ into $\langle f(v_i)\rangle^\omega$ and fixes every other pointwise. Since $\langle f(v_i)\rangle^\omega\cong \langle [a]_D\rangle^\omega$, the resulting permutation is an automorphism moving $[a]_D$ into $f[\F^\ue]$.
		Hence,  we obtain that  $\F\preceq\uf\preceq\, ^I\F/D$.
	\end{proof}

	\section{Discussion}
	Regarding the other type of ultrafilter extension,  \cite{Saveliev2} wonders about theories such that elementary equivalence of their models lifts to their
	ultrafilter extensions. Here, we have an immediate corollary:
	
	\begin{cor}
		For any bounded graphs $\F$ and $\mathfrak{G}$ if $\F\equiv \mathfrak{G}$, then $\uf\equiv \mathfrak{G}^\ue$.
	\end{cor}
	\noindent It is natural to ask what other type of infinite structures can be isolated within the class \[\mathsf{F}= \{\mathfrak{F}: \mathfrak{F}\equiv \uf\}\] besides  bounded graphs.

		We note that $\mathsf{F}$ is not an elementary class, since it is not closed under ultraproducts. To see this, consider the structures $\mathfrak{N}_n =\langle n, <\rangle$  with their standard orderings. Then for any non-principal ultrafilter $D$ over $\omega$, we have that $\prod_{n\in\omega}\mathfrak{N}_n/D \models  \forall x (x\not < x)$, but $\big(\prod_{n\in\omega}\mathfrak{N}_n/D\big)^\ue \models  \exists x( x < x)$, similarly to Example \ref{ex}. We wonder whether $\mathsf{F}$ is closed under elementary equivalence, that is given $\mathfrak{F}\in \mathsf{F}$ and $\mathfrak{F}\equiv \mathfrak{G}$ does it imply $\mathfrak{G}\in \mathsf{F}$?

		\subsubsection*{Acknowledgement}
		The author is grateful to Zalán Gyenis for all the pleasant conversations about the
		topic.  We thank the  anonymous reviewers for their many insightful comments that improved the paper substantially.
		Supported by the ÚNKP-20-3 New
		National Excellence Program of the Ministry for Innovation and Technology from
		the Source of the National Research, Development and Innovation Fund.

	\end{document}